\documentclass[dvipdfmx,reqno,12pt]{amsart}
\usepackage{amsmath,amssymb,amsthm}
\usepackage{bm,bbm,comment,mathtools}
\usepackage{booktabs}
\usepackage{longtable}
\usepackage{pdflscape}
\usepackage[table]{xcolor}
\usepackage{graphicx}
\usepackage[
 dvipdfmx,
 setpagesize=false,
 bookmarks=true,
 bookmarksdepth=1,
 bookmarksnumbered=true,
 colorlinks=false,
 hidelinks,
]{hyperref}
\usepackage{enumitem}

\setlist[enumerate]{
    leftmargin=1.5em,
    itemsep=0.3em,
    topsep=0.3em,
    parsep=0pt,
    partopsep=0pt
}
\newlength{\margins}
\usepackage[top=\margins,bottom=\margins,left=\margins,right=\margins]{geometry}

\allowdisplaybreaks
\makeatletter
\@namedef{subjclassname@2020}{%
\textup{2020} Mathematics Subject Classification}
\title[Fourier Coefficients of Siegel--Eisenstein Series of Weight \(m+1\)]
{Fourier Coefficients of Siegel--Eisenstein Series of Degree \(2m\) and Weight \(m+1\)}
\makeatother
\author[N.
Takeda]{Nobuki TAKEDA}

\address{Department of Mathematics,
Graduate School of Science,
Kyoto University,
Kyoto 606-8502,
Japan}
\email{takeda.nobuki.z04@kyoto-u.jp}

\theoremstyle{definition}
\newtheorem{dfn}{Definition}[section]
\newtheorem{rem}[dfn]{Remark}

\theoremstyle{plain}
\newtheorem{prop}[dfn]{Proposition}

\newtheorem{lem}[dfn]{Lemma}
\newtheorem{thm}[dfn]{Theorem}
\newtheorem{cor}[dfn]{Corollary}

\DeclareMathOperator{\tr}{Tr}

\usepackage{expl3}

\ExplSyntaxOn

\clist_const:Nn \c_alphabet_clist { A,B,C,D,E,F,G,H,I,J,K,L,M,N,O,P,Q,R,S,T,U,V,W,X,Y,Z,
a,b,c,d,e,f,g,h,i,j,k,l,m,n,o,p,q,r,s,t,u,v,w,x,y,z }

\cs_new:Npn \makealphabet #1 #2
  {
    \clist_map_inline:Nn \c_alphabet_clist
      { \exp_args:Nc \newcommand { #1 ##1 } { #2 { ##1 } } }
  }

\makealphabet{bb}{ \mathbb }
\makealphabet{bf}{ \mathbf }
\makealphabet{cal}{ \mathcal }
\makealphabet{frak}{ \mathfrak }
\makealphabet{rm}{ \mathrm }
\clist_map_inline:nn 
  {rank, disc, GL, Sp, Sym, prim, M, diag, reg, sq }
  { \exp_args:Nc \newcommand { #1 } { \mathrm { #1 } } }

\ExplSyntaxOff

\newcommand{\ord}{\operatorname{ord}}
\renewcommand{\Im}{\mathrm{Im}}
\renewcommand{\Re}{\mathrm{Re}}

\makeatletter
\newcommand{\sectionnotoc}[1]{%
  \begingroup
  \def\@tocwrite##1##2{}%
  \section*{#1}%
  \endgroup
}
\makeatother
\numberwithin{equation}{section}
\begin{document}

\begin{abstract}
    We study Fourier coefficients of the Siegel--Eisenstein series
    \(E_{m+1}^{(2m)}(Z)\) for \(m\equiv1\pmod4\) and \(m\geq5\).
    Using the Fourier expansion formula due to Mizumoto,
    we determine the constant term and the exceptional non-zero Fourier
    coefficients.
    The result gives a higher-degree analogue of the degree-two formulas of
    Kohnen and Nagaoka,
    which were later rederived by Haruki.
\end{abstract}

\maketitle

\tableofcontents
\section{Introduction}

This paper studies the Fourier expansion of the Siegel--Eisenstein series
of degree \(n=2m\) and boundary weight \(k=m+1=(n+2)/2\), where \(m\equiv 1 \pmod4\).
Throughout the paper,
we put \(e(x)=\exp(2\pi\sqrt{-1}x)\).
Here and below,
\(E_k^{(n)}(Z)\) denotes the value at \(s=0\) of the Eisenstein series
\(E_k^{(n)}(Z,s)\) defined in Section~\ref{sec:siegel-eisenstein},
whenever this value is defined.
We write
\[ E_k^{(n)}(Z)=\sum_{T\in\Lambda_n} A_T^{(n)}(Y)e(\tr(TX)), \qquad Z=X+\sqrt{-1}Y\in\bbH_n. \]
The aim is to determine the coefficient functions \(A_T^{(n)}(Y)\),
including the non-holomorphic dependence on \(Y\) which may occur at this
weight.

The holomorphy of Siegel--Eisenstein series in the small weight range was
clarified by Shimura~\cite{Shimura1983} and Weissauer~\cite{Weissauer1984}.
For even weights \(k\) in the range \(k\geq(n+1)/2\), the boundary
phenomena occur at
\[ k=\frac{n+2}{2}\equiv2\pmod4, \qquad k=\frac{n+3}{2}\equiv2\pmod4. \]
The second case is nearly holomorphic.
The present paper treats the first case.

The first non-holomorphic example in this direction occurs in degree \(2\),
where \(n=2\) and \(k=2\).
Nagaoka~\cite{Nagaoka1992} and Kohnen~\cite{Kohnen1993} described the Fourier expansion of \(E_2^{(2)}(Z,0)\).
Haruki~\cite{Haruki1997} later computed this expansion by using the general
Fourier expansion formula of Mizumoto~\cite{Mizumoto1993}.
We use Mizumoto's formula as the starting point.

The constant term can be calculated uniformly for all odd \(m\geq3\).
Putting \(\epsilon_m=(-1)^{(m+1)/2}\),
Theorem~\ref{thm:constant-term} gives
\[ A_0^{(2m)}(Y)=1+ \frac{(\epsilon_m-1)\pi^{m+1}} {\zeta(m+1)\zeta(2m)\sqrt{\det Y}} \left( \frac{\zeta(m)}{2^{m+1}m!}+ \frac{(m-2)!\zeta(m-1)}{(2m)!} \zeta_1^{(2m)}(Y,m-1) \right), \]
where \(\zeta_1^{(2m)}(Y,m-1)\) is understood by meromorphic continuation.
Thus \(A_0^{(2m)}(Y)=1\) if \(m\equiv3\pmod4\),
whereas the constant term has non-trivial \(Y\)-dependence if
\(m\equiv1\pmod4\).
The degree-six case \(m=3\) contains two additional middle-rank terms; they
cancel, as shown in
Proposition~\ref{prop:degree-six-middle-rank-cancellation}.

For positive semi-definite non-zero indices,
Proposition~\ref{prop:nonzero-coefficient-contribution} expresses
\(A_T^{(n)}(Y)\) as a finite sum over
\(\nu\in\mathcal J^{(n)}_{\lambda}\), where \(\lambda=\rank(T)\).
The exceptional terms come from the case \((\lambda,\nu)=(n-2,n)\)
and from the square-discriminant pole of the non-degenerate Siegel series.
For indices \(T\not\geq0\),
Proposition~\ref{prop:indefinite-index-coefficients} gives the possible
contributing terms.

The paper is organized as follows.
In Section~2,
we recall notation for Siegel modular forms and Eisenstein series.
In Section~3,
we recall the Koecher--Maass zeta function,
Siegel series,
confluent hypergeometric functions,
and the Fourier expansion formula due to Mizumoto~\cite{Mizumoto1993}.
In Section~4,
we calculate the constant term for all odd \(m\geq3\).
In Section~5,
we calculate the non-zero index terms in the main case \(m\equiv1\pmod4\).
In Section~6,
we assemble the results into the main theorem.
The appendix contains the middle-rank cancellation in degree \(6\).

\textbf{Acknowledgements.}
The author is grateful to T. Ikeda for his guidance and support,
and to H. Katsurada and S. Horinaga for valuable comments.
This work was supported by the Japan Science and Technology Agency (JST) SPRING Program,
Grant Number JPMJSP2110,
and by JSPS KAKENHI Grant Number JP26KJ1378.

\textbf{Notation.}
For a commutative ring \(R\),
let \(M_{a,b}(R)\) denote the set of \(a\times b\) matrices over \(R\),
and put \(M_a(R)=M_{a,a}(R)\).
Let \(\bbS_a(R)\) denote the set of symmetric \(a\times a\) matrices over \(R\).
Let \(\bbS_a(R)_{>0}\) (resp. \(\bbS_a(R)_{\ge0}\)) denote the subset of \(\bbS_a(R)\) consisting of positive definite (resp. positive semi-definite) matrices if \(R\) is an subring of \(\bbR\).
We write \(0_{a,b}\) for the zero \(a\times b\) matrix,
\(0_a=0_{a,a}\),
and \(I_a\) for the identity matrix of degree \(a\).
We denote Euler's constant by \(\gamma_E\).

Landau's notation is always used with respect to \(s\to0\).
In particular, in Laurent expansions at \(s=0\), \(O(s^r)\) denotes a function
holomorphic near \(s=0\) and vanishing to order at least \(r\), and \(O(1)\)
denotes a function holomorphic near \(s=0\).

\section{Siegel Modular Forms and Eisenstein Series}\label{sec:siegel-eisenstein}

Let \(\bbH_n\) be the Siegel upper half space of degree \(n\), that is,
\[ \bbH_n=\left\{ Z\in M_n(\bbC) \,\middle|\, Z={}^{t}\!Z=X+\sqrt{-1}Y,\ X\in \bbS_n(\bbR),\ Y\in \bbS_n(\bbR)_{>0} \right\}. \]
We put
\[ \Gamma_n=\Sp_n(\bbZ)=\left\{ g\in \GL_{2n}(\bbZ) \,\middle|\, ^{t}\!g J_n g=J_n \right\}, \]
where \(J_n=\begin{pmatrix} 0_n & I_n \\ -I_n & 0_n \end{pmatrix}\).

For
\[
    g=
    \begin{pmatrix}
        A & B \\
        C & D
    \end{pmatrix}
    \in \Sp_n(\bbR),
\]
and for a function \(F\) on \(\bbH_n\),
we define the slash operator of weight \(k\) by
\[ (F|_k g)(Z)=\det(CZ+D)^{-k}F\left((AZ+B)(CZ+D)^{-1}\right). \]
A holomorphic function \(F\) on \(\bbH_n\) is called a Siegel modular form of weight \(k\) with respect to \(\Gamma_n\) if
\(F|_k g=F\) for all  \(g\in\Gamma_n\).

Let
\[
    \Gamma_{n,\infty}=\left\{
    \begin{pmatrix}
        *       & * \\
        0_{n,n} & *
    \end{pmatrix}
    \in \Gamma_n
    \right\}.
\]
For a positive even integer \(k\),
we define the Siegel--Eisenstein series of degree \(n\) and weight \(k\) by
\[ E_k^{(n)}(Z,s)=\det(\Im Z)^s \sum_{g\in \Gamma_{n,\infty}\backslash\Gamma_n}\det(CZ+D)^{-k} \left|\det(CZ+D)\right|^{-2s}, \]
where
\(
g=
\begin{pmatrix}
    A & B \\
    C & D
\end{pmatrix}\).
This series converges absolutely for \(\Re(s)>(n+1-k)/2\)
and has meromorphic continuation to the whole \(s\)-plane.
Whenever the value at \(s=0\) is defined,
we write
\[ E_k^{(n)}(Z)=E_k^{(n)}(Z,0). \]

We recall the following result of Shimura~\cite{Shimura1983} and Weissauer~\cite{Weissauer1984}.

\begin{prop}\label{prop:eisenstein}
    Let \(k\) be a positive even integer.
    Suppose that \(k \geq (n+1)/2\).
    Then the following assertions hold.
    \begin{enumerate}
        \item
              The series \(E_k^{(n)}(Z,s)\) is holomorphic in \(s\) at \(s=0\).

        \item
              If neither
              \(k=(n+2)/2\equiv 2 \pmod 4\)
              nor
              \(k=(n+3)/2\equiv 2 \pmod 4\)
              holds,
              then \(E_k^{(n)}(Z)\) is a holomorphic modular form of weight \(k\).

        \item
              If
              \(k=(n+3)/2\equiv 2 \pmod 4\),
              then \(E_k^{(n)}(Z)\) is a nearly holomorphic modular form of weight \(k\).

              More precisely,
              there are holomorphic functions
              \(p_1\) and \(p_2\) on \(\bbH_n\) such that
              \[ E_k^{(n)}(Z)=\Delta\left\{ p_1(Z)+p_2(Z)\log\det(\Im Z) \right\}, \]
              where
              \[ \Delta=\det\left(\frac12(1+\delta_{ij})\frac{\partial}{\partial z_{ij}}\right). \]
              These functions may be chosen so that
              \begin{enumerate}
                  \item
                        \(\Delta p_1\) has a Fourier expansion with rational coefficients.

                  \item
                        \(\Delta p_2=0\).

                  \item
                        \(\pi^n p_2\) is a modular form of weight \((n-1)/2\)
                        with rational Fourier coefficients.
              \end{enumerate}
    \end{enumerate}
\end{prop}

The case \(k=\dfrac{n+2}{2}\equiv 2 \pmod 4\) is the other exceptional case.
In degree \(2\),
this is the case of \(E_2^{(2)}(Z,0)\),
whose Fourier expansion was studied by Kohnen~\cite{Kohnen1993} and Nagaoka~\cite{Nagaoka1992}.
Afterwards,
Haruki~\cite{Haruki1997} gave a more direct approach to the Fourier expansion of \(E_2^{(2)}(Z,s)\) by using the Fourier expansion due to Mizumoto~\cite{Mizumoto1993}.

We shall study this exceptional case by using Mizumoto's Fourier expansion.
For later use,
we first recall the relevant part of his Fourier expansion and fix notation.

\section{Auxiliary Functions and the Fourier Expansion Formula of Mizumoto}

In this section,
we review the Fourier expansion of the Siegel--Eisenstein series given by Mizumoto~\cite{Mizumoto1993},
which plays a central role
in our analysis.
We first recall the definitions of the Koecher--Maass zeta function,
the Siegel series,
and the confluent hypergeometric functions.

\subsection{Koecher--Maass Zeta Function}
Let \(M_{n,\nu}(\bbZ)^{\prim}\subset M_{n,\nu}(\bbZ)\) denote the subset
consisting of primitive matrices.
Here a matrix is called primitive if the cokernel of the induced map
\(\bbZ^\nu\to\bbZ^n\) is torsion-free.
For an integer \(1\leq\nu\leq n\) and a positive definite symmetric matrix
\(g\in\bbS_n(\bbR)_{>0}\),
we define the Koecher--Maass zeta function by
\[ \zeta_\nu^{(n)}(g,s)=\sum_{a\in M_{n,\nu}(\bbZ)^{\prim}/\GL_\nu(\bbZ)} \det(g[a])^{-s}, \]
where \(g[a]={}^t\!aga\).
We also recall the completed Riemann zeta function
\[ \xi(s)=\pi^{-s/2}\Gamma\left(\frac{s}{2}\right)\zeta(s). \]
Using these functions,
we define the completed Koecher--Maass zeta function by
\[ \Xi_\nu^{(n)}(g,s)=2\varepsilon_\nu(s)\varepsilon_\nu\left(\frac n2-s\right)R_\nu^{(n)}(g,s), \]
where
\[ R_\nu^{(n)}(g,s)=\left(\prod_{j=0}^{\nu-1}\xi(2s-j)\right)\zeta_\nu^{(n)}(g,s), \]
and
\[ \varepsilon_\nu(s)=\prod_{j=0}^{\nu-1}\left(s-\frac j2\right). \]
The Koecher--Maass zeta series \(\zeta_\nu^{(n)}(g,s)\) converges
absolutely for \(\operatorname{Re}(s)>n/2\).

The analytic continuation and the functional equation of
\(\Xi_\nu^{(n)}(g,s)\) were established by Koecher~\cite{Koecher1954}
and Maass~\cite{Maass1971}.

\begin{prop}
    The function \(\Xi_\nu^{(n)}(g,s)\) extends to an entire function of \(s\).
    Moreover,
    it satisfies the functional equation
    \[ \det(g)^{\nu/4}\Xi_\nu^{(n)}\left(g,\frac n2-s\right)=\det(g^{-1})^{\nu/4}\Xi_\nu^{(n)}(g^{-1},s). \]
\end{prop}

\begin{rem}[Adelic interpretation]
    Let \(G=\GL_n\), and let \(P_\nu\subset G\) be the standard parabolic
    subgroup of type \((\nu,n-\nu)\).
    Put \(e_\nu={}^t(1_\nu,0)\).
    Then \(G(\bbQ)/P_\nu(\bbQ)\) parametrizes rational \(\nu\)-planes in
    \(\bbQ^n\).
    Moreover, the map \(\gamma\mapsto \gamma e_\nu\) induces an
    identification
    \[ \GL_n(\bbZ)/(P_\nu(\bbQ)\cap\GL_n(\bbZ)) \simeq M_{n,\nu}(\bbZ)^{\prim}/\GL_\nu(\bbZ). \]

    Choose \(a_\infty\in\GL_n(\bbR)\) such that
    \(g={}^ta_\infty a_\infty\), and put
    \[ H_\nu(a_\infty)=\det({}^t(a_\infty e_\nu)(a_\infty e_\nu)). \]
    Let \(f_{\nu,s}\) be the standard spherical section of the degenerate
    principal series induced from \(P_\nu(\bbA)\).
    Thus \(f_{\nu,s}\) is right invariant under \(\GL_n(\widehat{\bbZ})\)
    at the finite places and under \(\mathrm O(n)\) at the real place, and
    is normalized by
    \[ f_{\nu,s}((1_f, a_\infty))=H_\nu(a_\infty)^{-s}. \]
    The associated Eisenstein series is
    \[ E_\nu(a,s)=\sum_{\gamma\in \GL_n(\bbQ)/P_\nu(\bbQ)} f_{\nu,s}(a\gamma). \]
    Evaluating at \(a=(1_f, a_\infty)\), we obtain
    \[ E_\nu((1_f, a_\infty),s)= \sum_{\gamma\in \GL_n(\bbZ)/(P_\nu(\bbQ)\cap\GL_n(\bbZ))} H_\nu(a_\infty\gamma)^{-s} =\zeta_\nu^{(n)}(g,s), \]
    since
    \(H_\nu(a_\infty\gamma)=\det(g[\gamma e_\nu])\).
    Hence the Koecher--Maass zeta function is the classical realization of a
    spherical degenerate Eisenstein series on \(\GL_n(\bbA)\).

    Under this identification, the normalizing factor used in \(\Xi_\nu^{(n)}(g,s)\)
    is the standard one attached to the spherical degenerate Eisenstein series.
    Thus the usual functional equation of the normalized Eisenstein series,
    associated with the Weyl element exchanging the two blocks of type \((\nu,n-\nu)\),
    specializes to the functional equation of \(\Xi_\nu^{(n)}(g,s)\). \(\diamond\)
\end{rem}

\begin{cor}\label{cor:koecher-maass-zeta-poles}
    Assume that \(n=2m\).
    For \(1\leq \nu\leq 2m\),
    put \(\tau(\nu)=(2m+1-\nu)/2\).
    Then the function \(\zeta_\nu^{(2m)}\left(g,2s+\tau(\nu)\right)\)
    has at most a simple pole at \(s=0\) if \(1\leq\nu\leq m\),
    and is holomorphic at \(s=0\) if \(m+1\leq\nu\leq 2m\).
\end{cor}

\begin{proof}
    By the definition of the completed Koecher--Maass zeta function,
    we have
    \[ \Xi_\nu^{(2m)}(g,u)=2\varepsilon_\nu(u)\varepsilon_\nu(m-u)\prod_{j=0}^{\nu-1}\xi(2u-j)\zeta_\nu^{(2m)}(g,u). \]
    Put \(u=2s+\tau(\nu)=2s+(2m+1-\nu)/2\).

    The factor \(\varepsilon_\nu(m-u)=\varepsilon_\nu\left((\nu-1)/2-2s\right)\)
    always has a simple zero at \(s=0\),
    and the factor \(\varepsilon_\nu(u)=\varepsilon_\nu\left(2s+(2m+1-\nu)/2\right)\)
    has a zero at \(s=0\) if and only if
    \[ \frac{2m+1-\nu}{2}=\frac a2 \]
    for some \(0\leq a\leq\nu-1\).
    This occurs exactly when \(\nu\geq m+1\).
    In that case the zero is simple.

    Now consider
    \[ \prod_{j=0}^{\nu-1} \xi(2u-j)= \prod_{j=0}^{\nu-1} \xi(4s+2m+1-\nu-j). \]
    If \(1\leq\nu\leq m\),
    then \(2m+1-\nu-j \geq 2m+2-2\nu \geq 2\).
    Hence all factors are holomorphic and non-zero at \(s=0\).

    If \(m+1\leq\nu\leq2m\),
    then the product contains exactly the two singular
    factors \(\xi(1+4s)\) and  \(\xi(4s)\).
    Thus the product of the \(\xi\)-factors has a pole of order \(2\) at \(s=0\).

    Combining these observations,
    the stated holomorphy and pole assertions follow.
\end{proof}

The following residue formula is due to Arakawa~\cite{Arakawa1990}.
We use it in the form recorded by Nagaoka~\cite[Proposition~3.7]{Nagaoka2024}.

\begin{prop}\label{prop:xi-special-values}
    Let \(M\) and \(\nu\) be positive integers with \(M\geq 2\nu-1\).
    Let \(g\in \bbS_M(\bbR)_{>0}\).
    For \(1\leq \mu\leq \nu-1\),
    put
    \[ c_{M,\nu,\mu}=- \frac{ \varepsilon_{\nu\mu}\, \varepsilon_\nu\left(\dfrac{M-\mu}{2}\right) v(\nu-\mu) }{ 2 \varepsilon_\mu\left(\dfrac{\nu}{2}\right) \varepsilon_\mu\left(\dfrac{M-\nu}{2}\right) }, \]
    where
    \[
        \varepsilon_{\nu\mu}=
        \prod_{\substack{0\leq a\leq \nu-1\\ a\neq \mu}}
        \left(\frac{\mu-a}{2}\right),
        \qquad
        v(a)=
        \begin{cases}
            \displaystyle\prod_{j=2}^{a}\xi(j), & a\geq2, \\[0.8em]
            1,                                  & a=1.
        \end{cases}
    \]
    Then
    \[ \Xi_\nu^{(M)}\left(g,\frac{\mu}{2}\right)=c_{M,\nu,\mu}\Xi_\mu^{(M)}\left(g,\frac{\nu}{2}\right), \]
    and
    \[ \Xi_\nu^{(M)}\left(g,\frac{M-\mu}{2}\right)=c_{M,\nu,\mu}\det(g)^{-\nu/2}\Xi_\mu^{(M)}\left(g^{-1},\frac{\nu}{2}\right). \]
\end{prop}

\begin{lem}\label{lem:complementary-rank-duality}
    Let \(S\in \bbS_n(\bbR)_{>0}\).
    For \(1\leq r\leq n-1\),
    we have
    \[ \zeta_r^{(n)}(S,u)=\det(S)^{-u} \zeta_{n-r}^{(n)}(S^{-1},u). \]
    This identity initially holds in the region of absolute convergence
    and extends to all \(u\in\bbC\) by meromorphic continuation.
\end{lem}

\begin{proof}
    This is the primitive version of the complementary identity of
    Arakawa~\cite[(2.13)]{Arakawa1990}.
    In the region of absolute convergence,
    a primitive rank \(r\) sublattice of \(\bbZ^n\) is sent bijectively to its primitive
    orthogonal complement of rank \(n-r\).
    If \(A\in M_{n,r}(\bbZ)^{\prim}\) represents the first sublattice and
    \(A^\ast\in M_{n,n-r}(\bbZ)^{\prim}\) represents the complement, then
    \({}^t\!AA^\ast=0\), and we may easily verify that
    \[ \det({}^t\!ASA)=\det(S)\det({}^t\!A^\ast S^{-1}A^\ast). \]
    Hence
    \begin{align*}
        \zeta_r^{(n)}(S,u) & =\sum_{A\in M_{n,r}(\bbZ)^{\prim}/\GL_r(\bbZ)} \det({}^t\!ASA)^{-u}                                         \\
                           & =\det(S)^{-u} \sum_{A^\ast\in M_{n,n-r}(\bbZ)^{\prim}/\GL_{n-r}(\bbZ)} \det({}^t\!A^\ast S^{-1}A^\ast)^{-u} \\
                           & =\det(S)^{-u}\zeta_{n-r}^{(n)}(S^{-1},u).
    \end{align*}
    Meromorphic continuation gives the identity for all \(u\in\bbC\).
\end{proof}

\subsection{Siegel Series}
For a rational symmetric matrix \(T\in\bbS_\nu(\bbQ)\),
the Siegel series
\(S_\nu(T,t)\) is defined by
\[ S_\nu(T,t)=\sum_{R\in\bbS_\nu(\bbQ)/\bbS_\nu(\bbZ)} d(R)^{-t}e(\tr(TR)). \]
Here \(d(R)\) denotes the absolute value of the determinant of a denominator
matrix of \(R\).
More precisely,
write \(R=C^{-1}D\),
where \(C,D\in M_\nu(\bbZ)\) form a
coprime symmetric pair,
that is,
\(\det C\neq0\), \(C{}^tD=D{}^tC\),
and \(C\) and \(D\) are right coprime,
meaning that there exist \(U,V\in M_\nu(\bbZ)\) such that
\(CU+DV=I_\nu\).
Then \(d(R)=|\det C|\).

For a non-degenerate symmetric matrix \(h\in\bbS_\lambda(\bbQ)\) of even degree,
we use the discriminant convention
\[ \disc(h)=(-1)^{\lambda/2}\det h\in\bbQ^\times. \]
We also use the same notation for its square class in
\(\bbQ^\times/\bbQ^{\times2}\).
When we say that the discriminant of \(h\) is a square,
we mean that \(\disc(h)\in\bbQ^{\times2}\).
The corresponding quadratic character is the Kronecker character associated
with this square class.

The following proposition follows from the results of
B{\"o}cherer~\cite{Boecherer1984} and Kitaoka~\cite{Kitaoka1984}.
\begin{prop}\label{prop:siegel-series-factorization}
    The following assertions hold.
    \begin{enumerate}
        \item
              Let \(T\in\bbS_\nu(\bbQ)\) be non-degenerate.
              Then there exists a family \((F_p(T,X))_p\) of polynomials,
              indexed by primes \(p\),
              such that \(F_p(T,X)=1\) for all but finitely many \(p\),
              and
              \[
                  S_\nu(T,s)=
                  \begin{cases}
                      \displaystyle \zeta(s)^{-1} \prod_{j=1}^{\nu/2}\zeta(2s-2j)^{-1} L\left(s-\frac{\nu}{2},\chi_T\right) \prod_p F_p(T,p^{-s}), & \text{if \(\nu\) is even}, \\[1.2em]
                      \displaystyle \zeta(s)^{-1} \prod_{j=1}^{(\nu-1)/2}\zeta(2s-2j)^{-1} \prod_p F_p(T,p^{-s}),                                  & \text{if \(\nu\) is odd}.
                  \end{cases}
              \]
              Here \(\chi_T\) denotes the quadratic character attached to \(\disc(T)=(-1)^{\nu/2}\det T\) in the convention above.

        \item     Let \(h\in\bbS_\lambda(\bbQ)\) and let \(\lambda<\nu\).
              Then
              \[ S_\nu(\diag(h,0_{\nu-\lambda}),s)=\calZ_{\lambda,\nu}(s) S_\lambda(h,s-\nu+\lambda), \]
              where
              \[ \calZ_{\lambda,\nu}(s)= \frac{\zeta(s+\lambda-\nu)}{\zeta(s)} \prod_{j=1}^{\nu-\lambda}\frac{\zeta(2s-\nu-j)}{\zeta(2s-2j)}. \]
              In particular,
              we have
              \[ S_\nu(0_\nu,s)=\frac{\zeta(s-\nu)}{\zeta(s)} \prod_{j=1}^{\nu}\frac{\zeta(2s-\nu-j)}{\zeta(2s-2j)}. \]
    \end{enumerate}
\end{prop}

The local Siegel polynomials \(F_p(T,X)\) were subsequently calculated explicitly by Katsurada~\cite{Katsurada1999}.

\subsection{Confluent Hypergeometric Function}
We now recall the confluent hypergeometric functions on tube domains,
following Shimura~\cite{Shimura1982}.
We put
\(
\kappa(\lambda)=(\lambda+1)/2.
\)
For \(g\in \bbS_\lambda(\bbR)_{>0}\),
\(h\in \bbS_\lambda(\bbR)\),
and
\((\alpha,\beta)\in\bbC^2\),
define
\[
    \eta_\lambda(g,h;\alpha,\beta)=\int_{\substack{x\in \bbS_\lambda(\bbR)\\ x\pm h>0}}
    e^{-\tr(gx)}\det(x+h)^{\alpha-\kappa(\lambda)}\det(x-h)^{\beta-\kappa(\lambda)}\,dx,
\]
where
\(dx=\prod_{1\leq i\leq j\leq \lambda}dx_{ij}\).
This integral converges for \(\Re(\alpha)>\kappa(\lambda)-1\),
\(\Re(\beta)>\lambda\).
Following the normalization used by Shimura,
we set
\[ \eta_\lambda^\ast(g,h;\alpha,\beta)=\det(g)^{\alpha+\beta-\kappa(\lambda)} \eta_\lambda(g,h;\alpha,\beta). \]
The function \(\eta_\lambda^\ast(g,h;\alpha,\beta)\) admits meromorphic
continuation in \((\alpha,\beta)\).

\subsection{Fourier Expansion of Siegel--Eisenstein Series}

We denote by \(\Lambda_n\) the set of half-integral symmetric matrices of degree
\(n\):
\[ \Lambda_n= \left\{ T=(t_{ij})\in\bbS_n(\bbQ) \;\middle|\; t_{ii}\in\bbZ,\ 2t_{ij}\in\bbZ\text{ for }i<j \right\}. \]
We also put
\[ \Lambda_n^\ast= \{T\in\Lambda_n\mid \det T\neq0\}, \qquad \Lambda_n^{(r)}= \{T\in\Lambda_n\mid \rank T=r\}. \]
For a symmetric matrix \(T\) and an appropriate matrix \(a\),
we use the shorthand \(T[a]={}^t\!aTa\).
Finally,
we put
\[ \Gamma_\nu(s)=\pi^{\nu(\nu-1)/4}\prod_{j=0}^{\nu-1}\Gamma\left(s-\frac j2\right). \]

We now state the explicit Fourier expansion formula due to Mizumoto
for the Siegel--Eisenstein series.

\begin{prop}[{\cite[Theorem 1.8]{Mizumoto1993}}]\label{prop:mizumoto}
    Let \(n\in\bbZ_{>0}\),
    \(k\in2\bbZ_{\geq0}\),
    and \(\Re(s)>n\).
    Then the Siegel--Eisenstein series admits the Fourier expansion
    \[ E_k^{(n)}(Z,s)=\sum_{\nu=0}^{n}\sum_{\lambda=0}^{\nu}F_{\nu,\lambda}^{(n)}(Z,s). \]

    For \(\nu=0,\ldots,n\),
    setting \(\kappa(\nu)=(\nu+1)/2\),
    we have
    \begin{align*}
        F_{\nu,0}^{(n)}(Z,s) & =(-1)^{k\nu/2}2^\nu\pi^{\nu\kappa(\nu)} \frac{\Gamma_\nu(2s+k-\kappa(\nu))}{\Gamma_\nu(s)\Gamma_\nu(s+k)} S_\nu(0,2s+k)\det(Y)^s \zeta_\nu^{(n)}\left(2Y,2s+k-\kappa(\nu)\right).
    \end{align*}

    For \(1\leq\lambda\leq\nu\leq n\),
    we have
    \[ F_{\nu,\lambda}^{(n)}(Z,s)= \sum_{h\in\Lambda_\lambda^\ast} \sum_{r\in M_{n,\lambda}(\bbZ)^{\prim}/\GL_\lambda(\bbZ)} b_{\nu,\lambda}^{(n)} \bigl(h[{}^tr],Y,s\bigr) e\bigl(\tr(h[{}^tr]X)\bigr), \]
    where
    \begin{align*}
        b_{\nu,\lambda}^{(n)}
        (h[{}^tr],Y,s)
         & =(-1)^{k\nu/2}
        2^\nu
        \pi^{\nu\kappa(\nu)+\lambda(\nu-\lambda)/2}
        \\
         & \quad\times
        \frac{\Gamma_{\nu-\lambda}(2s+k-\kappa(\nu))}
        {\Gamma_\nu(s)\Gamma_\nu(s+k)}
        S_\nu
        \bigl(\diag(h,0_{\nu-\lambda}),2s+k\bigr)
        \\
         & \quad\times
        \det(Y)^s
        \det(2Y[r])^{\kappa(\nu)-k-2s}
        \\
         & \quad\times
        \eta_\lambda^\ast
        \left(
        2Y[r],
        \pi h;
        s+k+\frac{\lambda-\nu}{2},
        s+\frac{\lambda-\nu}{2}
        \right)
        \\
         & \quad\times
        \zeta_{\nu-\lambda}^{(n-\lambda)}
        \left(
        2g(Y,u_r),
        2s+k-\kappa(\nu)
        \right).
    \end{align*}

    Here \(u_r\in\GL_n(\bbZ)\) is chosen so that its first \(\lambda\)
    columns are given by \(r\).
    Writing \(u_r=(r\ r_1)\) with
    \(r_1\in M_{n,n-\lambda}(\bbZ)\),
    the term \(g(Y,u_r)\) is the Schur complement of \(Y[r]\) in \(Y[u_r]\):
    \[ g(Y,u_r)=Y[r_1]-{}^tr_1Yr\,(Y[r])^{-1}{}^trYr_1. \]
\end{prop}

Let
\[ E_k^{(n)}(Z,s)=\sum_{T\in\Lambda_n} A_T^{(n)}(Y,s)e(\tr(TX)) \]
be the Fourier expansion of \(E_k^{(n)}(Z,s)\),
where \(Z=X+\sqrt{-1}Y\).
We shall calculate these coefficients at \(s=0\),
and we put
\[ A_T^{(n)}(Y)= A_T^{(n)}(Y,0). \]
We note that the coefficient \(A_T^{(n)}(Y,s)\) is holomorphic at \(s=0\).

\begin{rem}[Adelic interpretation of Mizumoto's expansion]
    Mizumoto's formula can be regarded as the classical expression of the
    Fourier expansion of a Siegel Eisenstein series on the adelic symplectic
    group.
    Let \(G=\Sp_n\), and let \(P=MN\subset G\) be the Siegel parabolic
    subgroup.
    Thus \(M\simeq\GL_n\), and
    \[
        N=\left\{
        n(U)=
        \begin{pmatrix}
            1_n & U   \\
            0   & 1_n
        \end{pmatrix}
        \;\middle|\;
        U={}^t\!U
        \right\}.
    \]
    Let \(K_f=\Sp_n(\widehat{\bbZ})\).
    We take the standard \(K_f\)-spherical section at the finite places and
    the archimedean section of weight \(k\).
    The associated adelic Eisenstein series is
    \[ E(g,s)=\sum_{\gamma\in P(\bbQ)\backslash G(\bbQ)} f_{k,s}(\gamma g). \]
    We normalize \(f_{k,s}\) so that
    \[ E(g_Z,s)=E_k^{(n)}(Z,s), \qquad g_Z=(g_\infty(Z),1_f), \]
    where \(g_\infty(Z)\in\Sp_n(\bbR)\) sends
    \(\sqrt{-1}\cdot1_n\) to \(Z=X+\sqrt{-1}Y\).

    For \(T\in\Lambda_n\), define a character of \(N(\bbA)\) by
    \[ \psi_T(n(U))=\psi(\tr(TU)), \]
    where \(\psi\) is the standard additive character of \(\bbA/\bbQ\).
    With the self-dual Haar measure attached to \(\psi\), the adelic
    \(T\)-th Fourier coefficient is
    \[ E_T(g,s)=\int_{N(\bbQ)\backslash N(\bbA)} E(n(U)g,s)\psi_T(n(U))^{-1}\,dU. \]
    Evaluating at \(g=g_Z\), this coefficient has the classical form
    \[ E_T(g_Z,s)=A_T^{(n)}(Y,s)e(\tr(TX)). \]

    In the region of absolute convergence, we insert the Eisenstein series
    into this integral and unfold.
    The unfolding is organized by the Bruhat decomposition with respect to
    the Siegel parabolic:
    \[ G(\bbQ)=\coprod_{\nu=0}^n P(\bbQ)w_\nu P(\bbQ), \]
    where the cell \(P(\bbQ)w_\nu P(\bbQ)\) is characterized by the
    condition that the lower-left block has rank \(\nu\).
    Thus the Fourier coefficient decomposes as
    \[ E_T(g,s)=\sum_{\nu=0}^n E_{T,\nu}(g,s), \]
    where \(E_{T,\nu}(g,s)\) denotes the contribution of the Bruhat cell of
    rank \(\nu\).

    Suppose now that \(T\) has rank \(\lambda\).
    We write
    \[ T=h[{}^tr], \qquad h\in\Lambda_\lambda^\ast, \quad r\in M_{n,\lambda}(\bbZ)^{\prim}. \]
    The rank-\(\nu\) contribution vanishes unless \(\lambda\leq\nu\).
    For \(\lambda\leq\nu\), the unfolded contribution
    \(E_{T,\nu}(g_Z,s)\) is the \(T\)-th Fourier coefficient of the
    \((\nu,\lambda)\)-term in Mizumoto's expansion.
    Therefore Mizumoto's expansion
    \[ E_k^{(n)}(Z,s)=\sum_{\nu=0}^{n}\sum_{\lambda=0}^{\nu}F_{\nu,\lambda}^{(n)}(Z,s) \]
    is the classical form of the adelic Fourier expansion obtained by
    unfolding the Siegel Eisenstein series and decomposing the result
    according to the Bruhat rank \(\nu\). \(\diamond\)
\end{rem}

\subsection{Fourier Coefficients in the Holomorphic Range}

Before specializing to the boundary weight,
we recall the standard Fourier coefficient formula in the holomorphic range.
This formula serves as a reference point for the local Siegel polynomial factors
which occur in the Siegel series.

For a commutative ring \(R\) and symmetric matrices \(A,B\in\bbS_n(R)\),
we write \(A\sim_R B\)
if there exists \(g\in\GL_n(R)\) such that \(B=A[g]\).
When \(A\) and \(B\) have different degrees,
we use this notation only after both matrices have been embedded into
matrices of the same degree by adjoining zero blocks.

Let \(T\in\Lambda_n\) be positive semi-definite of rank \(\ell\).
For each prime \(p\),
choose a non-degenerate matrix
\(\widetilde T_p\in\Lambda_\ell\otimes\bbZ_p\) such that
\(T\sim_{\bbZ_p} \diag(\widetilde T_p,
0_{n-\ell})\).
The local Siegel polynomial \(F_p(\widetilde T_p,X)\) is independent of this
choice.
We put \(F_p^\ast(T,X)=F_p(\widetilde T_p,X)\).
If \(\ell=0\),
we use the convention \(F_p^\ast(0_n,X)=1\).

For \(T\in\Lambda_n\) with \(T\geq0\) and \(\rank(T)=\ell\),
choose a positive definite matrix \(\widetilde T\in\Lambda_\ell\) such that \(T\sim_{\bbZ}\diag(\widetilde T,
0_{n-\ell})\).
If \(\ell\) is even,
let \(\chi_T^\ast\) be the quadratic character attached to
\(\bbQ\bigl(\sqrt{(-1)^{\ell/2}\det\widetilde T}\bigr)\).
This character is independent of the choice of \(\widetilde T\).

We put
\[\widetilde E_k^{(n)}(Z)=\zeta(1-k)\prod_{i=1}^{\lfloor n/2\rfloor}\zeta(1+2j-2k)E_k^{(n)}(Z).\]

\begin{prop}\label{prop:holomorphic-range-fourier-coefficients}
    Let \(k\in2\bbZ\).
    Assume that \(k\geq (n+1)/2\),
    and that neither of the exceptional conditions
    \(k=(n+2)/2\equiv2\pmod4\),
    \(k=(n+3)/2\equiv2\pmod4\)
    holds.
    Then \(\widetilde E_k^{(n)}(Z)\) is holomorphic,
    and for \(T\in\Lambda_n\),
    \(T\geq0\),
    \(\rank(T)=\ell\),
    the Fourier coefficient is
    \[ a(T,\widetilde E_k^{(n)})= 2^{\left\lfloor(\ell+1)/2\right\rfloor} \prod_{p\mid\det(2\widetilde T)} F_p^\ast(T,p^{k-\ell-1}) \calL^{(n)}_{k,\ell}(T), \]
    where
    \[
        \calL^{(n)}_{k,\ell}(T)=
        \begin{cases}
            \displaystyle \prod_{i=\ell/2+1}^{\lfloor n/2\rfloor} \zeta(1+2i-2k)\, L(1+\ell/2-k,\chi_T^\ast), & \ell\text{ even}, \\[1.2em]
            \displaystyle \prod_{i=(\ell+1)/2}^{\lfloor n/2\rfloor} \zeta(1+2i-2k),                           & \ell\text{ odd}.
        \end{cases}
    \]
    Here an empty product is understood to be \(1\).
\end{prop}

\begin{proof}
    This is the standard Fourier coefficient formula for holomorphic
    Siegel--Eisenstein series.
    The formula follows from the Fourier coefficient formula for Eisenstein series
    due to Shimura~\cite{Shimura1983} and the explicit factorization of the local Siegel
    series into the polynomials \(F_p(T,X)\) due to Katsurada~\cite{Katsurada1999}.
\end{proof}

\section{The Constant Term}

In this section,
we calculate the constant term \(A_0^{(n)}(Y)\) of
\(E_k^{(n)}(Z)\).
Throughout this section,
we assume that \(n=2m\) with odd \(m\geq3\),
and put \(k=m+1\).

The main case of the paper is \(m\equiv1\pmod4\).
Nevertheless,
the constant term can be calculated uniformly for all odd \(m\geq3\).
The same formula also explains the simplification which occurs when
\(m\equiv3\pmod4\).
The only extra contribution which appears when \(m=3\) comes from the
middle ranks \(\nu=3,4\);
its cancellation is recorded separately in
Appendix~\ref{app:degree-six-middle-rank}.

We put
\[ \tau(\nu)=k-\kappa(\nu)=\frac{n+1-\nu}{2}. \]

By Proposition~\ref{prop:mizumoto},
the constant Fourier coefficient is given by
\[ A_0^{(n)}(Y,s)= \sum_{\nu=0}^{n}C_\nu(Y,s), \qquad C_\nu(Y,s)=F_{\nu,0}^{(n)}(Z,s). \]
We expand each \(C_\nu(Y,s)\) as a Laurent series at \(s=0\),
writing
\[ C_\nu(Y,s)=\sum_j c_{\nu,j}(Y)s^j. \]
Since \(E_k^{(n)}(Z,s)\) is holomorphic at \(s=0\),
the desired constant term is obtained by taking the finite part:
\[ A_0^{(n)}(Y)=\sum_{\nu=0}^{n}c_{\nu,0}(Y). \]

For \(\nu=0\),
we have \(C_0(Y,s)=\det(Y)^s\).
Hence
\[ C_0(Y,0)=1. \]

For \(1\leq\nu\leq n\),
Proposition~\ref{prop:mizumoto} gives
\begin{equation}\label{eq:Cnu-first-column} C_\nu(Y,s)=(-1)^{k\nu/2}2^\nu\pi^{\nu\kappa(\nu)} \frac{\Gamma_\nu(2s+\tau(\nu))}{\Gamma_\nu(s)\Gamma_\nu(s+m+1)} S_\nu(0_\nu,2s+m+1)\det(Y)^s\zeta_\nu^{(n)}(2Y,2s+\tau(\nu)). \end{equation}
By Proposition~\ref{prop:siegel-series-factorization},
we have
\begin{equation}\label{eq:siegel-zero} S_\nu(0_\nu,w)=\frac{\zeta(w-\nu)}{\zeta(w)} \prod_{j=1}^{\nu}\frac{\zeta(2w-\nu-j)}{\zeta(2w-2j)}. \end{equation}

We first determine which terms can contribute to the finite part at \(s=0\).
Combining \eqref{eq:siegel-zero} with the order of the gamma quotient and
Corollary~\ref{cor:koecher-maass-zeta-poles},
we obtain
\[
    \ord_{s=0} C_\nu(Y,s)
    \geq
    \begin{cases}
        \left\lceil\nu/2\right\rceil-1,      & 1\leq\nu\leq m-1,                     \\[4pt]
        (m-3)/2,                             & \nu=m,m+1,                            \\[4pt]
        \left\lceil(2m-\nu)/2\right\rceil-1, & m+2\leq\nu\leq2m-1,\ \nu\text{ odd},  \\[4pt]
        \left\lceil(2m-\nu)/2\right\rceil,   & m+2\leq\nu\leq2m-1,\ \nu\text{ even}, \\[4pt]
        0,                                   & \nu=2m.
    \end{cases}
\]
Here the lower bound is obtained by separating the three elementary contributions:
the zero of the reciprocal gamma quotient,
the zeros and poles of \(S_\nu(0,2s+m+1)\),
and the possible poles of \(\zeta_\nu^{(2m)}(Y,s+m+1-\nu/2)\).
The borderline equality in the middle ranks occurs only for \(m=3\).

Consequently,
if \(m\geq 5\),
then the only terms that can contribute to the finite part are
\(C_0\), \(C_1\), \(C_2\), \(C_{2m-1}\), and \(C_{2m}\).
If \(m=3\),
the additional terms \(C_3\) and \(C_4\) may also contribute;
their sum is zero by
Proposition~\ref{prop:degree-six-middle-rank-cancellation}.

We put \(\epsilon_m=(-1)^{(m+1)/2}\).

\subsection*{The case \(\nu=1\)}

By \eqref{eq:Cnu-first-column},
we have
\[ C_1(Y,s)=\epsilon_m 2\pi \frac{\Gamma(2s+m)}{\Gamma(s)\Gamma(s+m+1)} S_1(0,2s+m+1)\det(Y)^s\zeta_1^{(n)}(2Y,2s+m). \]
Using \eqref{eq:siegel-zero},
we obtain
\[ S_1(0,2s+m+1)=\frac{\zeta(m)}{\zeta(m+1)} +O(s). \]

We use the residue formula for the Epstein zeta function
\cite[\S 1.4,
    Theorem~1]{Terras1985}.
For a positive definite real symmetric matrix \(S\) of degree \(n\),
put
\[ Z(S,u)=\frac{1}{2}\sum_{0\neq x\in\bbZ^n}({}^txSx)^{-u}. \]
With this normalization,
we have
\begin{equation}\label{eq:zeta-one-residue} \operatorname*{Res}_{u=n/2}Z(S,u)= \frac{\pi^{n/2}}{2\Gamma(n/2)\sqrt{\det S}}. \end{equation}
Since
\[ Z(S,u)=\zeta(2u)\zeta_1^{(n)}(S,u), \]
we obtain
\[ \operatorname*{Res}_{u=m}\zeta_1^{(n)}(2Y,u)=\frac{\pi^m}{2^{m+1}\Gamma(m)\zeta(2m)\sqrt{\det Y}}. \]
Hence
\[ \zeta_1^{(n)}(2Y,2s+m)= \frac{1}{2s} \frac{\pi^m}{2^{m+1}\Gamma(m)\zeta(2m)\sqrt{\det Y}} +O(1). \]

Moreover,
\[ \frac{\Gamma(2s+m)}{\Gamma(s)\Gamma(s+m+1)}=s\frac{\Gamma(m)}{\Gamma(m+1)} +O(s^2). \]
Substituting these expansions,
we obtain
\[ C_1(Y,0)= \epsilon_m \frac{\pi^{m+1}\zeta(m)}{2^{m+1}m!\zeta(m+1)\zeta(2m)} \frac{1}{\sqrt{\det Y}}. \]

\subsection*{The case \(\nu=2\)}

By \eqref{eq:Cnu-first-column},
we have
\[ C_2(Y,s)= 4\pi^3 \frac{\Gamma_2\left(2s+\frac{2m-1}{2}\right)}{\Gamma_2(s)\Gamma_2(s+m+1)} S_2(0,2s+m+1) \det(Y)^s \zeta_2^{(n)} \left(2Y,2s+\frac{2m-1}{2}\right). \]
Using \eqref{eq:siegel-zero},
we have
\[ S_2(0,2s+m+1)=\frac{\zeta(m-1)\zeta(2m-1)}{\zeta(m+1)\zeta(2m)} +O(s). \]
Using the completed Koecher--Maass zeta function,
we have
\[ \zeta_2^{(n)}(2Y,u)=\frac{\Xi_2^{(n)}(2Y,u)}{2\varepsilon_2(u)\varepsilon_2(m-u)\xi(2u)\xi(2u-1)}. \]
A direct Laurent expansion gives
\[ \zeta_2^{(n)} \left(2Y,2s+\frac{2m-1}{2}\right)= -\frac{2\,\Xi_2^{(n)}\left(2Y,\frac{2m-1}{2}\right)}{s(2m-1)(2m-2)\xi(2m-1)\xi(2m-2)} +O(1). \]
Also,
\[ \frac{\Gamma_2\left(2s+\frac{2m-1}{2}\right)}{\Gamma_2(s)\Gamma_2(s+m+1)}= -\frac{s}{2\pi}\frac{\Gamma_2\left(\frac{2m-1}{2}\right)}{\Gamma_2(m+1)} +O(s^2). \]
Thus
\[ C_2(Y,0)= \frac{2^{2m+2}\pi^{2m}}{(2m)!(2m-1)(2m-2)} \frac{\zeta(m-1)}{\zeta(m+1)\zeta(2m-2)\zeta(2m)} \Xi_2^{(n)} \left(2Y,\frac{2m-1}{2}\right). \]

\subsection*{The case \(\nu=2m\)}

Since
\[ \zeta_{2m}^{(2m)}(g,u)=\det(g)^{-u}, \]
we have
\[ \zeta_{2m}^{(2m)} \left(2Y,2s+\frac12\right)= \det(2Y)^{-2s-1/2}. \]
Substituting this into \eqref{eq:Cnu-first-column},
\[ C_{2m}(Y,s)= 2^{-4ms+m}\pi^{m(2m+1)} \det(Y)^{-s-1/2} \frac{\Gamma_{2m}\left(2s+\frac12\right)}{\Gamma_{2m}(s)\Gamma_{2m}(s+m+1)} S_{2m}(0,2s+m+1). \]
By the definition of the multivariate gamma function,
\[ \frac{\Gamma_{2m}\left(2s+\frac12\right)}{\Gamma_{2m}(s)\Gamma_{2m}(s+m+1)}= -\frac{(2m)!}{2^{3m}m!\Gamma_{2m}(m+1)} +O(s). \]
Using \eqref{eq:siegel-zero},
we obtain
\[ S_{2m}(0,2s+m+1)=\frac{\zeta(2s+1-m)}{\zeta(2s+m+1)} \frac{\prod_{r=1}^{m}\zeta(4s+3-2r)}{\prod_{r=1}^{m}\zeta(4s+2r)}. \]
Since \(m\) is odd,
the factor \(\zeta(2s+1-m)\) has a simple zero at
\(s=0\),
while \(\zeta(4s+1)\) has a simple pole.
Thus
\[ S_{2m}(0,2s+m+1)= \frac{\zeta'(1-m)}{2\zeta(m+1)} \frac{\prod_{r=1}^{m-1}\zeta(1-2r)}{\prod_{r=1}^{m}\zeta(2r)} +O(s). \]
By the functional equation of the Riemann zeta function, since \(m\) is odd,
\[ \zeta'(1-m)=(-1)^{(m-1)/2}2^{-m}\pi^{1-m}\Gamma(m)\zeta(m). \]
Combining this identity with the gamma factor and the remaining zeta factors,
we obtain
\[ C_{2m}(Y,0)=-\frac{\pi^{m+1}\zeta(m)}{2^{m+1}m!\zeta(m+1)\zeta(2m)} \frac{1}{\sqrt{\det Y}}. \]

\subsection*{The case \(\nu=2m-1\)}

By \eqref{eq:Cnu-first-column},
we have
\begin{align*}
    C_{2m-1}(Y,s)
    = & (-1)^{(m+1)(2m-1)/2}2^{2m-1}\pi^{(2m-1)\kappa(2m-1)}                    \\
      & \times \frac{\Gamma_{2m-1}(2s+1)}{\Gamma_{2m-1}(s)\Gamma_{2m-1}(s+m+1)} \\
      & \times S_{2m-1}(0,2s+m+1)\det(Y)^s\zeta_{2m-1}^{(2m)}(2Y,2s+1).
\end{align*}
The gamma factor satisfies
\[ \frac{\Gamma_{2m-1}(2s+1)}{\Gamma_{2m-1}(s)\Gamma_{2m-1}(s+m+1)} =s\frac{(2m-2)!}{2^{3(m-1)}\Gamma_{2m-1}(m+1)}+O(s^2). \]
Using \eqref{eq:siegel-zero},
we obtain
\[ S_{2m-1}(0,2s+m+1)=\frac{1}{4s}\frac{\zeta(2-m)}{\zeta(m+1)} \frac{\prod_{r=1}^{m-2}\zeta(1-2r)}{\prod_{r=2}^{m}\zeta(2r)}+O(1). \]
Hence the zero of the gamma factor cancels the pole of
\(S_{2m-1}(0,2s+m+1)\) at \(s=0\).
Evaluating the remaining normalizing factors at \(s=0\), and using the
functional equation of the Riemann zeta function to rewrite the negative
integer zeta values, we get
\[ C_{2m-1}(Y,0)= \epsilon_m\frac{2^{6m-5}\pi^{2m}}{(2m-4)!(2m-2)!(2m)!} \frac{\zeta(m-1)}{\zeta(m+1)\zeta(2m-2)\zeta(2m)} \frac{\Xi_{2m-1}^{(2m)}(2Y,1)}{\prod_{j=2}^{2m-3}\xi(j)}. \]
\subsection*{The constant term}

Combining the preceding computations and the complementary-rank duality,
we obtain the following formula.

\begin{thm}\label{thm:constant-term}
    Assume that \(m\geq3\) is odd,
    and put
    \[ \epsilon_m=(-1)^{(m+1)/2}. \]
    Then the constant Fourier coefficient of \(E_{m+1}^{(2m)}(Z)\) is given by
    \begin{equation} A_0^{(2m)}(Y)=1+ \frac{(\epsilon_m-1)\pi^{m+1}} {\zeta(m+1)\zeta(2m)\sqrt{\det Y}} \left( \frac{\zeta(m)}{2^{m+1}m!}+ \frac{(m-2)!\zeta(m-1)}{(2m)!} \zeta_1^{(2m)}(Y,m-1) \right). \label{eq:constant-term-reduced} \end{equation}
    Here \(\zeta_1^{(2m)}(Y,m-1)\) is understood by meromorphic continuation.
\end{thm}

\begin{proof}
    Summing
    the contributions of \(C_0\), \(C_1\), \(C_2\), \(C_{2m-1}\), and \(C_{2m}\),
    we obtain

    \begin{align}
        A_0^{(2m)}(Y)
         & =1+ \frac{(\epsilon_m-1)\pi^{m+1}\zeta(m)} {2^{m+1}m!\zeta(m+1)\zeta(2m)} \frac{1}{\sqrt{\det Y}} \notag                                                                                                        \\
         & \quad+ \frac{2^{2m+2}\pi^{2m}} {(2m)!(2m-1)(2m-2)} \frac{\zeta(m-1)} {\zeta(m+1)\zeta(2m-2)\zeta(2m)} \Xi_2^{(2m)} \left(2Y,\frac{2m-1}{2}\right) \notag                                                        \\
         & \quad+ \epsilon_m \frac{2^{6m-5}\pi^{2m}} {(2m-4)!(2m-2)!(2m)!} \frac{\zeta(m-1)} {\zeta(m+1)\zeta(2m-2)\zeta(2m)} \frac{\Xi_{2m-1}^{(2m)}(2Y,1)} {\prod_{j=2}^{2m-3}\xi(j)}. \label{eq:constant-term-expanded}
    \end{align}
    When \(m=3\),
    the additional middle-rank terms \(C_3\) and \(C_4\) occur,
    but their sum is zero by
    Proposition~\ref{prop:degree-six-middle-rank-cancellation}.

    By the residue formula for the Epstein zeta function \cite[\S~1.4, Theorem~1]{Terras1985},
    we obtain
    \[ \Xi_2^{(2m)}\left(2Y,\frac{2m-1}{2}\right)= -\frac{(2m-1)(2m-2)}{2^{2m+2}}\xi(2m-2) \frac{\zeta_1^{(2m)}(Y,m-1)}{\sqrt{\det Y}}. \]
    By Lemma~\ref{lem:complementary-rank-duality} and the functional equation
    for the rank-one completed Koecher--Maass zeta function,
    we also have
    \[ \frac{\Xi_{2m-1}^{(2m)}(2Y,1)}{\prod_{j=2}^{2m-3}\xi(j)}= \frac{(2m-4)!(2m-2)!}{2^{6m-5}}\xi(2m-2) \frac{\zeta_1^{(2m)}(Y,m-1)}{\sqrt{\det Y}}. \]
    Substituting these identities into \eqref{eq:constant-term-expanded},
    we obtain the reduced formula \eqref{eq:constant-term-reduced}.
\end{proof}

\begin{rem}
    If \(m\equiv3\pmod4\), then \(\epsilon_m=1\), and  \(A_0^{(2m)}(Y)=1\).
    This is compatible with the fact that \(E_k^{(n)}(Z)\) is holomorphic in this case. \(\diamond\)
\end{rem}

\begin{rem}[The degree two case]
    When \(n=2\),
    the boundary weight is \(k=2\).
    This case is exceptional in a stronger sense than the higher degree cases
    considered in this paper.
    In fact,
    writing
    \[
        Y=
        \begin{pmatrix}
            y_{11} & y_{12} \\
            y_{12} & y_{22}
        \end{pmatrix}
        \in\bbS_2(\bbR)_{>0},
    \]
    the constant Fourier coefficient of \(E_2^{(2)}(Z,0)\) is
    \[ A_0^{(2)}(Y)=1- \frac{18}{\pi^2\sqrt{\det Y}} \left( 1+\frac{\gamma_E}{2} +\frac12\log\frac{y_{22}}{4\pi}- \log \left| \eta\left( \frac{y_{12}+\sqrt{-1}\sqrt{\det Y}}{y_{22}} \right) \right|^2 \right), \]
    where \(\gamma_E\) is Euler's constant and \(\eta\) is the Dedekind eta
    function.
    The logarithmic term comes from the first Kronecker limit formula.
    Thus the degree-two formula is not a specialization of
    Theorem~\ref{thm:constant-term};
    the zero-pole pattern in Mizumoto's expansion is different when \(n=2\). \(\diamond\)
\end{rem}

\section{Non-zero Index Terms}
\label{sec:nonzero-coefficients}

In this section,
we calculate the Fourier coefficients indexed by non-zero matrices in the exceptional case.
Throughout the section,
we assume that \(n=2m\), \(m\geq5\), \(m\equiv1\pmod4\).
Thus \(k=m+1\equiv2\pmod4\),
and \(n\geq10\).

\subsection{Index Sets and Scalar Factors}

Before treating the Fourier coefficients for non-zero indices,
we define the index sets and scalar factors used below.

For \(1\leq\lambda<\nu\leq n\),
put \(d=\nu-\lambda\).
We define
\begin{align*}
    \calC^{(n)}_{\lambda,\nu}(s)
    = & \Gamma_d\left(2s+\frac{n+1-\nu}{2}\right)\cdot\Gamma_d(s)^{-1}                                                    \\
      & \times\left( \varepsilon_d\left(2s+\frac{n+1-\nu}{2}\right) \varepsilon_d\left(-2s-\frac{\lambda+1-\nu}{2}\right)
    \prod_{j=0}^{d-1}\xi(4s+n+1-\nu-j) \right)^{-1}.
\end{align*}
We also put
\[ \calD^{(n)}_{\lambda,\nu}(s)=\calC^{(n)}_{\lambda,\nu}(s) \calZ_{\lambda,\nu}(2s+m+1), \]
where
\[ \calZ_{\lambda,\nu}(s)= \frac{\zeta(s+\lambda-\nu)}{\zeta(s)} \prod_{j=1}^{\nu-\lambda}\frac{\zeta(2s-\nu-j)}{\zeta(2s-2j)}. \]
Finally,
set
\[ \rho^{(n)}_{\lambda,\nu}=\ord_{s=0}\calD^{(n)}_{\lambda,\nu}(s). \]

\begin{lem}\label{lem:scalar-factor-order}
    We have \(\rho^{(n)}_{\lambda,\nu}\leq0\) precisely in the following cases:
    \[
        \begin{array}{c|c|c}
            \lambda              & \nu                           & \rho^{(n)}_{\lambda,\nu} \\ \hline
            n-1                  & n                             & 0                        \\
            n-2                  & n-1                           & 0                        \\
            n-2                  & n                             & -1                       \\
            n-3                  & n-2,n-1,n                     & 0                        \\
            n-4                  & n-3,n-2,n-1,n                 & 0                        \\
            n-5,\ n-6            & \lambda+1,\lambda+2,\lambda+4 & 0                        \\
            2\leq\lambda\leq n-7 & \lambda+1,\lambda+2           & 0                        \\
            1                    & 2,3,n                         & 0
        \end{array}
    \]
    In degree \(n=10\),
    one additionally has
    \[ \rho^{(10)}_{\lambda,\lambda+5}=0 \qquad (2\leq\lambda\leq5). \]
    Moreover,
    \[ \rho^{(n)}_{n-6,n-3}=1,\qquad \rho^{(n)}_{n-8,n-4}=1, \]
    whenever these pairs occur.
    All remaining pairs have positive order.
\end{lem}

\begin{proof}
    The assertion follows by counting the zeros and poles of the factors defining
    \(\calD^{(n)}_{\lambda,\nu}(s)\).
    The possible poles come from the zeta factors whose arguments become \(1\) at \(s=0\),
    while the zeros come from zeta factors whose arguments become negative even integers and from gamma factors.
    Combining these elementary order counts gives exactly the table above.
\end{proof}

We define \(\calI_{\lambda}^{\reg}=\calI_{n,\lambda}^{\reg}\) to be the set of
all integers \(\nu\) with \(\lambda<\nu\leq n\) such that
\((\lambda,\nu)\) appears in the table of
Lemma~\ref{lem:scalar-factor-order} and satisfies
\(\rho^{(n)}_{\lambda,\nu}\leq 0\).
In degree \(n=10\), we include the additional cases
\(\nu=\lambda+5\) stated in the same lemma.

We shall also need to keep the exceptional order-one cases in which
the zero of the scalar factor may be cancelled by the pole of the
non-degenerate Siegel series.
For this purpose, put
\[
    \calI_{\lambda}^{\sq}=
    \begin{cases}
        \{n-3\},   & \lambda=n-6,      \\
        \{n-4\},   & \lambda=n-8,      \\
        \emptyset, & \text{otherwise}.
    \end{cases}
\]
We then set
\[ \calI_{\lambda}^{(n)}= \calI_{\lambda}^{\reg}\cup \calI_{\lambda}^{\sq} \qquad (1\leq\lambda<n), \]
and finally define
\[
    \calJ_{\lambda}^{(n)}=
    \begin{cases}
        \{n\},                                & \lambda=n,      \\
        \{\lambda\}\cup\calI_{\lambda}^{(n)}, & 1\leq\lambda<n.
    \end{cases}
\]

\subsection{Positive Semi-definite Indices}

We now assume that \(T\geq0\) (positive semi-definite)
and \(T\neq0\).
Let \(\lambda=\rank(T)\).
If \(\lambda=n\),
then \(T>0\) (positive definite),
and we set \(h=T\) and \(r=1_n\).
If \(\lambda<n\),
choose \(h\in\Lambda_\lambda^\ast\) and \(r\in M_{n,\lambda}(\bbZ)^{\prim}\)
such that \(T=h[{}^tr]\).
The formulas below are independent of this choice.
In both cases,
we have \(h>0\).
For \(\lambda\leq\nu\leq n\),
put \(u_{\lambda,\nu}=m+1-\nu+\lambda\).

We denote by \(A_{T,\nu}^{(n)}(Y,s)\) the contribution of
\(F_{\nu,\lambda}^{(n)}(Z,s)\) to the \(T\)-th Fourier coefficient.
Thus
\[ A_T^{(n)}(Y,s)=\sum_{\nu=\lambda}^{n}A_{T,\nu}^{(n)}(Y,s). \]

\begin{prop}\label{prop:semidefinite-vanishing}
    Assume that \(T\geq0\),
    \(T\neq0\),
    and \(\rank(T)=\lambda<n\).
    If \(\lambda<\nu\leq n\) and \(\nu\notin\calI_{\lambda}^{(n)}\),
    then
    \[ A_{T,\nu}^{(n)}(Y,s)=O(s) \quad (s\to0). \]
    Hence this term does not contribute to the constant term at \(s=0\).
\end{prop}

\begin{proof}
    The \((\nu,\lambda)\)-term factors into the non-degenerate Siegel series
    \(S_\lambda(h,2s+u_{\lambda,\nu})\) and a remaining factor whose zeta part is
    \(\calD^{(n)}_{\lambda,\nu}(s)\).
    If \(\nu\notin\calI_{n,\lambda}^{\reg}\),
    then Lemma~\ref{lem:scalar-factor-order} gives
    \[ \ord_{s=0}\calD^{(n)}_{\lambda,\nu}(s)>0. \]

    The only possible pole of \(S_\lambda(h,2s+u_{\lambda,\nu})\) at \(s=0\) occurs when \(\lambda\) is even,
    the discriminant of \(h\) is a square,
    and \(u_{\lambda,\nu}=\lambda/2+1\).
    This condition is equivalent to \(\nu=m+\lambda/2\).
    Among the pairs for which \(\ord_{s=0}\calD^{(n)}_{\lambda,\nu}(s)=1\),
    this happens precisely for \((\lambda,\nu)=(n-6,n-3)\) or \((\lambda,\nu)=(n-8,n-4)\).
    These are exactly the indices added in \(\calI_{n,\lambda}^{\sq}\).
    Therefore,
    if \(\nu\notin\calI_{\lambda}^{(n)}\),
    the term remains \(O(s)\),
    even after taking into account the possible pole of the non-degenerate Siegel series.
\end{proof}

For each \(\nu\in\calJ_{\lambda}^{(n)}\),
we define \(\Phi_{\lambda,\nu}^{(n)}(Y,h,r;s)\) by extracting the Siegel series
\(S_\lambda(h,2s+u_{\lambda,\nu})\) from the \((\nu,\lambda)\)-term.

For \(\nu=\lambda\),
set
\[ \Phi_{\lambda,\lambda}^{(n)}(Y,h,r;s)=(-1)^{k\lambda/2}2^\lambda\pi^{\lambda\kappa(\lambda)} \frac{\det(Y)^s}{\Gamma_\lambda(s)\Gamma_\lambda(s+m+1)} \eta_\lambda \left( 2Y[r], \pi h; s+m+1, s \right). \]
For \(\lambda<\nu\),
set
\begin{align*}
    B_{\lambda,\nu}^{(n)}(s)
     & =(-1)^{k\nu/2}2^{\nu-1}\pi^{\nu\kappa(\nu)} \frac{\Gamma_{\nu-\lambda}\left(2s+\frac{n+1-\nu}{2}\right)}{\Gamma_{\nu-\lambda}(s)} \frac{1}{\Gamma_\nu(s+m+1)} \\
     & \quad\times \varepsilon_{\nu-\lambda}\left(2s+\frac{n+1-\nu}{2}\right)^{-1} \varepsilon_{\nu-\lambda}\left(-2s-\frac{\lambda+1-\nu}{2}\right)^{-1}            \\
     & \quad\times \prod_{j=0}^{\nu-\lambda-1} \xi(4s+n+1-\nu-j)^{-1},
\end{align*}
and define
\begin{align*}
    \Phi_{\lambda,\nu}^{(n)}(Y,h,r;s)
     & =B_{\lambda,\nu}^{(n)}(s) \calZ_{\lambda,\nu}(2s+m+1)\det(Y)^s\det(2Y[r])^{(\lambda-\nu)/2}                                                                                \\
     & \quad\times \frac{\eta_\lambda \left( 2Y[r], \pi h; s+m+1+\frac{\lambda-\nu}{2}, s+\frac{\lambda-\nu}{2} \right)} {\Gamma_\lambda \left( s+\frac{\lambda-\nu}{2} \right) } \\
     & \quad\times \Xi_{\nu-\lambda}^{(n-\lambda)} \left( 2g(Y,u_r), 2s+\frac{n+1-\nu}{2} \right).
\end{align*}
With this notation,
\[ A_{T,\nu}^{(n)}(Y,s)=\Phi_{\lambda,\nu}^{(n)}(Y,h,r;s) S_\lambda(h,2s+u_{\lambda,\nu}) \]
for all \(\nu\in\calJ_{\lambda}^{(n)}\).

If \(\Phi_{\lambda,\nu}^{(n)}(Y,h,r;s)\) is holomorphic at \(s=0\),
we write
\[ \Phi_{\lambda,\nu}^{(n)}(Y,h,r)=\Phi_{\lambda,\nu}^{(n)}(Y,h,r;0). \]
If \((\lambda,\nu)=(n-2,n)\),
then \(\Phi_{\lambda,\nu}^{(n)}(Y,h,r;s)\) has a simple pole at \(s=0\).
We write
\[ \Phi_{n-2,n}^{(n)}(Y,h,r;s)= \frac{\Phi_{n-2,n}^{(n),-1}(Y,h,r)}{s}+\Phi_{n-2,n}^{(n),0}(Y,h,r)+O(s). \]
If
\[ \Phi_{\lambda,\nu}^{(n)}(Y,h,r;s)=s\Phi_{\lambda,\nu}^{(n),1}(Y,h,r)+O(s^2), \]
we use this notation for the leading coefficient.

We also need the following notation for the possible pole of the non-degenerate Siegel series.
Assume that \(\lambda\) is even.
Let \(\chi_h\) be the quadratic character attached to \(\disc(h)\).
Near \(u=\lambda/2+1\),
we write
\[ S_\lambda(h,u)=L\left(u-\frac{\lambda}{2},\chi_h\right)\calP_h(u), \]
where \(\calP_h(u)\) is holomorphic at \(u=\lambda/2+1\).
If the discriminant of \(h\) is a square with our convention,
then
\[ S_\lambda(h,u)=\zeta\left(u-\frac{\lambda}{2}\right)\calP_h(u). \]

\begin{prop}\label{prop:nonzero-coefficient-contribution}
    Assume that \(T\geq0\),
    \(T\neq0\),
    and \(\rank(T)=\lambda\).
    Let \(h\) and \(r\) be as above.
    Then
    \[ A_T^{(n)}(Y)=\sum_{\nu\in\calJ_{\lambda}^{(n)}} A_{T,\nu}^{(n)}(Y). \]
    For \(\nu\in\calJ_{\lambda}^{(n)}\),
    the term \(A_{T,\nu}^{(n)}(Y)\) is given as follows.

    \begin{enumerate}
        \item
              If \((\lambda,\nu)=(n-2,n)\),
              then
              \[ A_{T,n}^{(n)}(Y)=\Phi_{n-2,n}^{(n),0}(Y,h,r)S_{n-2}(h,m-1) +2\Phi_{n-2,n}^{(n),-1}(Y,h,r)S_{n-2}'(h,m-1). \]

        \item
              Suppose that \(\lambda\) is even,
              \(\disc(h)\) is a square,
              and \(\nu=m+\lambda/2\).
              If \(n-\lambda=0,2,4\),
              then
              \begin{align*}
                  A_{T,\nu}^{(n)}(Y)
                   & =\Phi_{\lambda,\nu}^{(n)}(Y,h,r) \left\{ \calP_h'(u_{\lambda,\nu})+ \gamma_E\calP_h(u_{\lambda,\nu}) \right\} \\
                   & \quad+ \frac12 \left. \frac{d}{ds} \Phi_{\lambda,\nu}^{(n)}(Y,h,r;s) \right|_{s=0} \calP_h(u_{\lambda,\nu}).
              \end{align*}
              If \(n-\lambda=6,8\),
              then
              \[ A_{T,\nu}^{(n)}(Y)=\frac12\Phi_{\lambda,\nu}^{(n),1}(Y,h,r)\calP_h(u_{\lambda,\nu}). \]
              If \(n-\lambda\geq10\),
              then
              \[ A_{T,\nu}^{(n)}(Y)=0. \]

        \item
              In all remaining cases with \(\nu\in\calJ_{\lambda}^{(n)}\),
              we have
              \[ A_{T,\nu}^{(n)}(Y)=\Phi_{\lambda,\nu}^{(n)}(Y,h,r) S_\lambda(h,u_{\lambda,\nu}). \]
    \end{enumerate}
\end{prop}

\begin{proof}
    If \(\lambda=n\),
    only the term with \(\nu=n\) occurs.
    If \(\lambda<n\),
    Proposition~\ref{prop:semidefinite-vanishing} shows that only the terms with
    \(\nu\in\calJ_{\lambda}^{(n)}\) can contribute.

    For such \(\nu\),
    we have
    \[ A_{T,\nu}^{(n)}(Y,s)=\Phi_{\lambda,\nu}^{(n)}(Y,h,r;s) S_\lambda(h,2s+u_{\lambda,\nu}). \]

    If \((\lambda,\nu)=(n-2,n)\),
    substituting the Laurent expansion of
    \(\Phi_{n-2,n}^{(n)}(Y,h,r;s)\)
    and the Taylor expansion of \(S_{n-2}(h,m-1+2s)\)
    gives the formula in \((1)\).

    Suppose next that the square-discriminant pole condition holds.
    Then
    \[ S_\lambda(h,2s+u_{\lambda,\nu})= \frac{\calP_h(u_{\lambda,\nu})}{2s} +\calP_h'(u_{\lambda,\nu})+\gamma_E\calP_h(u_{\lambda,\nu})+O(s). \]
    If \(n-\lambda=0,2,4\),
    the factor \(\Phi_{\lambda,\nu}^{(n)}(Y,h,r;s)\) is holomorphic at \(s=0\),
    and extracting the coefficient of \(s^0\) gives the first formula in \((2)\).
    If \(n-\lambda=6,8\),
    the same extraction using
    \[ \Phi_{\lambda,\nu}^{(n)}(Y,h,r;s)=s\Phi_{\lambda,\nu}^{(n),1}(Y,h,r)+O(s^2) \]
    gives the second formula in \(2\).
    If \(n-\lambda\geq10\),
    the zero of \(\Phi_{\lambda,\nu}^{(n)}(Y,h,r;s)\) has order at least \(2\),
    so the term gives no constant term.

    In the remaining cases,
    both factors are holomorphic at \(s=0\),
    and direct substitution gives the formula in \((3)\).
\end{proof}

\subsection{Indices That Are Not Positive Semi-definite}

We now treat the Fourier coefficients indexed by matrices that are not positive semi-definite.
Let \(T\in\Lambda_n\) be non-zero,
and suppose that \(T\not\geq0\).
Put \(\lambda=\rank(T)\),
and choose \(h\in\Lambda_\lambda^\ast\) and \(r\in M_{n,\lambda}(\bbZ)^{\prim}\) such that
\(T=h[{}^tr]\).
Then \(h\) is non-degenerate and not positive definite.
Let the signature of \(h\) be \((p,q)\),
where \(p+q=\lambda\) and \(q>0\).
For \(\lambda\leq\nu\leq n\),
put
\(u_{\lambda,\nu}=m+1-\nu+\lambda\).

We first rewrite the factor containing the confluent hypergeometric function.
By the formula for the confluent hypergeometric function due to Shimura~\cite[Theorem~4.2 and (4.6.K)]{Shimura1982},
there exists a function \(H_{p,q,\nu-\lambda}(s;Y,h,r)\),
holomorphic at \(s=0\),
such that
\begin{align}
     & \frac{ \eta_\lambda \left( 2Y[r], \pi h; s+m+1-\frac{\nu-\lambda}{2}, s-\frac{\nu-\lambda}{2} \right) }{ \Gamma_\lambda\left(s-\frac{\nu-\lambda}{2}\right) } \notag                                                                   \\
     & \qquad=H_{p,q,\nu-\lambda}(s;Y,h,r) \frac{ \Gamma_p\left(s-\frac{\nu-\lambda+q}{2}\right) \Gamma_q\left(s+m+1-\frac{\nu-\lambda+p}{2}\right) }{ \Gamma_\lambda\left(s-\frac{\nu-\lambda}{2}\right) }. \label{eq:indefinite-eta-factor}
\end{align}
Here \(\Gamma_0(s)=1\).

The second gamma factor in the numerator of \eqref{eq:indefinite-eta-factor}
is holomorphic and non-zero at \(s=0\).
Indeed,
for \(0\leq a\leq q-1\),
we have
\[ m+1-\frac{\nu-\lambda+p+a}{2}=m+1-\frac{\nu-q+a}{2}\geq\frac32. \]
We put
\[ \delta_q(\nu-\lambda)=\ord_{s=0} \frac{\Gamma_p\left(s-\frac{\nu-\lambda+q}{2}\right)}{\Gamma_\lambda\left(s-\frac{\nu-\lambda}{2}\right)}. \]
This number is independent of \(p\),
and is given by
\begin{equation}\label{eq:delta-q-d}
    \delta_q(\nu-\lambda)=
    \begin{cases}
        \left\lceil q/2\right\rceil,   & \nu-\lambda\equiv0\pmod2, \\[4pt]
        \left\lfloor q/2\right\rfloor, & \nu-\lambda\equiv1\pmod2.
    \end{cases}
\end{equation}

If \(\lambda<\nu\),
define
\begin{align*}
    \Psi_{\lambda,\nu}^{(n),p,q}(Y,h,r;s)
     & :=B_{\lambda,\nu}^{(n)}(s) \calZ_{\lambda,\nu}(2s+m+1)\det(Y)^s\det(2Y[r])^{(\lambda-\nu)/2}                                                                                                              \\
     & \quad\times H_{p,q,\nu-\lambda}(s;Y,h,r) \frac{ \Gamma_p\left(s-\frac{\nu-\lambda+q}{2}\right) \Gamma_q\left(s+m+1-\frac{\nu-\lambda+p}{2}\right) }{ \Gamma_\lambda\left(s-\frac{\nu-\lambda}{2}\right) } \\
     & \quad\times \Xi_{\nu-\lambda}^{(n-\lambda)} \left( 2g(Y,u_r), 2s+\frac{n+1-\nu}{2} \right),
\end{align*}
and if \(\nu=\lambda\),
define
\begin{align*}
    \Psi_{\lambda,\lambda}^{(n),p,q}(Y,h,r;s)
     & =(-1)^{k\lambda/2}2^\lambda\pi^{\lambda\kappa(\lambda)} \det(Y)^s\frac{H_{p,q,0}(s;Y,h,r)}{\Gamma_\lambda(s+m+1)} \\
     & \quad\times \frac{ \Gamma_p\left(s-\frac q2\right) \Gamma_q\left(s+m+1-\frac p2\right) }{ \Gamma_\lambda(s) }.
\end{align*}
Then
\[ A_{T,\nu}^{(n)}(Y,s)=\Psi_{\lambda,\nu}^{(n),p,q}(Y,h,r;s) S_\lambda(h,2s+u_{\lambda,\nu}). \]

\begin{lem}\label{lem:square-discriminant-sign}
    Assume that \(\lambda\) is even.
    Let \(h\in\Lambda_\lambda^\ast\) have signature \((p,q)\).
    If \(\disc(h)\in\bbQ^{\times2}\),
    then
    \[ q\equiv\lambda/2\pmod2. \]
\end{lem}

\begin{proof}
    Since \(h\) has signature \((p,q)\),
    we have \(\operatorname{sgn}(\det h)=(-1)^q\).
    Hence
    \[ \operatorname{sgn}(\disc(h))= \operatorname{sgn}\left((-1)^{\lambda/2}\det h\right)= (-1)^{\lambda/2+q}. \]
    If \(\disc(h)\) is a square in \(\bbQ^\times\),
    then \(\disc(h)>0\).
    Therefore \(\lambda/2+q\) is even.
\end{proof}

\begin{prop}\label{prop:indefinite-index-coefficients}
    Let \(T\in\Lambda_n\) be non-zero and suppose that \(T\not\geq0\).
    Then
    \[ A_T^{(n)}(Y)=\sum_{\nu=\lambda}^{n}A_{T,\nu}^{(n)}(Y). \]
    We say that \((\lambda,\nu)\) is in the square-discriminant case if
    all of the following conditions hold:
    (i) \(\lambda\) is even,
    (ii) \(\disc(h)\in\bbQ^{\times2}\),
    (iii) \(\nu=m+\lambda/2\).
    The summands are described as follows.

    \begin{enumerate}
        \item
              Suppose first that \((\lambda,\nu)\) is not in the
              square-discriminant case.
              Then \(A_{T,\nu}^{(n)}(Y)\neq 0\)
              can occur only if \((q,\lambda,\nu)\) is one of the following
              cases:
              \[
                  \begin{array}{c|c}
                      q & (\lambda,\nu)                                               \\ \hline
                      1 & (\lambda,\lambda+1),\ (n-2,n),\ (n-3,n),\ (n-4,n-1),\ (1,n) \\[2pt]
                      2 & (n-2,n).
                  \end{array}
              \]
              When \(n=10\), the case \(q=1\) also includes \((2,7)\), \((3,8)\), \((4,9)\), and \((5,10)\).
              In these cases, write
              \[ \Psi_{\lambda,\nu}^{(n),p,q}(Y,h,r;s)= \frac{\Psi_{\lambda,\nu}^{(n),p,q;-1}(Y,h,r)}{s} +\Psi_{\lambda,\nu}^{(n),p,q;0}(Y,h,r)+O(s). \]
              Then
              \[ A_{T,\nu}^{(n)}(Y)= \Psi_{\lambda,\nu}^{(n),p,q;0}(Y,h,r)S_\lambda(h,u_{\lambda,\nu})+ 2\Psi_{\lambda,\nu}^{(n),p,q;-1}(Y,h,r) S_\lambda'(h,u_{\lambda,\nu}). \]

        \item
              Suppose next that \((\lambda,\nu)\) is in the
              square-discriminant case.
              Then \(A_{T,\nu}^{(n)}(Y)\neq 0\)
              can occur only if \((q,\lambda,\nu)\) is one of the following
              cases:
              \[
                  \begin{array}{c|c}
                      q & (\lambda,\nu)     \\ \hline
                      1 & (n,n),\ (n-4,n-2) \\
                      2 & (n-2,n-1).
                  \end{array}
              \]
              Near \(u=u_{\lambda,\nu}\), write
              \[ S_\lambda(h,u)=\zeta\left(u-\frac{\lambda}{2}\right)\calP_h(u), \]
              where \(\calP_h(u)\) is holomorphic at \(u=u_{\lambda,\nu}\), and put
              \[ Q_h(u_{\lambda,\nu})= \calP_h'(u_{\lambda,\nu})+\gamma_E\calP_h(u_{\lambda,\nu}). \]
              If
              \[ \Psi_{\lambda,\nu}^{(n),p,q}(Y,h,r;s)= \Psi_{\lambda,\nu}^{(n),p,q;0}(Y,h,r)+ s\Psi_{\lambda,\nu}^{(n),p,q;1}(Y,h,r) +O(s^2), \]
              then
              \[ A_{T,\nu}^{(n)}(Y)= \Psi_{\lambda,\nu}^{(n),p,q;0}(Y,h,r)Q_h(u_{\lambda,\nu}) +\frac{1}{2}\Psi_{\lambda,\nu}^{(n),p,q;1}(Y,h,r)\calP_h(u_{\lambda,\nu}). \]
    \end{enumerate}
    Moreover,
    every term \(A_{T,\nu}^{(n)}(Y,s)\) is holomorphic at \(s=0\).
\end{prop}

\begin{proof}
    We first compute the finite parts.
    The identity
    \[ A_{T,\nu}^{(n)}(Y,s)= \Psi_{\lambda,\nu}^{(n),p,q}(Y,h,r;s) S_\lambda(h,2s+u_{\lambda,\nu}) \]
    reduces the calculation to multiplying Laurent expansions.

    \begin{enumerate}
        \item
              Outside the square-discriminant case, the Siegel series is
              holomorphic at \(u=u_{\lambda,\nu}\).
              Hence
              \[ S_\lambda(h,2s+u_{\lambda,\nu})= S_\lambda(h,u_{\lambda,\nu})+2sS_\lambda'(h,u_{\lambda,\nu})+O(s^2). \]
              Multiplying this expansion by the Laurent expansion of
              \(\Psi_{\lambda,\nu}^{(n),p,q}(Y,h,r;s)\) gives the formula in
              \((1)\).

        \item
              In the square-discriminant case, we have
              \[ S_\lambda(h,2s+u_{\lambda,\nu})= \frac{\calP_h(u_{\lambda,\nu})}{2s}+Q_h(u_{\lambda,\nu})+O(s). \]
              Multiplying this expansion by the Taylor expansion of
              \(\Psi_{\lambda,\nu}^{(n),p,q}(Y,h,r;s)\) gives the formula in
              \((2)\).
    \end{enumerate}

    It remains to determine when a finite part can occur.
    Combining Lemma~\ref{lem:scalar-factor-order},
    \eqref{eq:delta-q-d}, and
    Lemma~\ref{lem:square-discriminant-sign}, we obtain the following
    restrictions.

    \begin{enumerate}
        \item
              Outside the square-discriminant case, a finite part can occur
              only if
              \[ \rho^{(n)}_{\lambda,\nu}+\delta_q(\nu-\lambda)\leq 0. \]
              Since \(\delta_q(\nu-\lambda)=0\) exactly when \(q=1\) and
              \(\nu-\lambda\) is odd, the pairs with
              \(\rho^{(n)}_{\lambda,\nu}=0\) give the first row of the table.
              The unique pair with \(\rho^{(n)}_{\lambda,\nu}=-1\), namely
              \((n-2,n)\), also allows \(q=2\).

        \item
              In the square-discriminant case, the condition
              \(u_{\lambda,\nu}=\lambda/2+1\) gives
              \[ (\lambda,\nu)=(n,n),\ (n-2,n-1),\ (n-4,n-2),\ldots. \]
              Moreover, Lemma~\ref{lem:square-discriminant-sign} imposes \(q\equiv \lambda/2 \pmod 2\).
              Since the Siegel series has at most a simple pole, a finite part
              can occur only if
              \(\rho^{(n)}_{\lambda,\nu} + \delta_q(\nu-\lambda) \leq 1\)
              for \(\lambda<\nu\).
              For \(\lambda=\nu=n\), the corresponding condition is \(\delta_q(0)\leq 1\).
              These restrictions leave exactly
              \[ (n,n),\ q=1, \qquad (n-2,n-1),\ q=2, \qquad (n-4,n-2),\ q=1. \]
    \end{enumerate}

    Finally, we show that no principal part occurs.
    Outside the square-discriminant case, a principal part could occur
    only if \(\rho^{(n)}_{\lambda,\nu}+\delta_q(\nu-\lambda)=-1\).
    The only possible pair with \(\rho^{(n)}_{\lambda,\nu}=-1\) is
    \((n-2,n)\), but then \(\delta_q(2)=\left\lceil q/2\right\rceil\geq1\).
    Hence no principal part occurs in this case.

    In the square-discriminant case, a principal part would require
    \(\rho^{(n)}_{\lambda,\nu}+\delta_q(\nu-\lambda)=0\)
    for \(\lambda<\nu\), or \(\delta_q(0)=0\) for \(\lambda=\nu=n\).
    The latter is impossible because \(q>0\).
    The former has only the formal possibility \((\lambda,\nu,q)=(n-2,n-1,1)\),
    but this is ruled out by Lemma~\ref{lem:square-discriminant-sign}.
    Thus no principal part occurs.
\end{proof}

\section{Main Theorem}

We now collect the results of the preceding sections.
Throughout this section,
let \(n=2m\), \(k=m+1\), \(\epsilon_m=(-1)^{(m+1)/2}\), where \(m\geq3\) is odd.
We write
\[ E_{m+1}^{(2m)}(Z)=\sum_{T\in\Lambda_n}A_T^{(n)}(Y)e(\tr(TX)). \]

\begin{thm}\label{thm:final-fourier-coefficients}
    The constant Fourier coefficient is
    \[ A_0^{(2m)}(Y)=1+ \frac{(\epsilon_m-1)\pi^{m+1}} {\zeta(m+1)\zeta(2m)\sqrt{\det Y}} \left( \frac{\zeta(m)}{2^{m+1}m!}+ \frac{(m-2)!\zeta(m-1)}{(2m)!} \zeta_1^{(2m)}(Y,m-1) \right), \]
    where \(\zeta_1^{(2m)}(Y,m-1)\) is understood by meromorphic continuation.

    Assume now that \(m\equiv1\pmod4\).
    For every non-zero \(T\in\Lambda_n\), the coefficient
    \(A_T^{(n)}(Y)\) is determined as follows.

    \begin{enumerate}
        \item Suppose that \(T\geq0\), and put
              \(\lambda=\rank(T)\).
              Choose \(h\) and \(r\) as in Section~\ref{sec:nonzero-coefficients}.
              Then
              \[ A_T^{(n)}(Y)=\sum_{\nu\in\calJ_\lambda^{(n)}} A_{T,\nu}^{(n)}(Y). \]
              The summands are given by
              Proposition~\ref{prop:nonzero-coefficient-contribution}.
              They fall into the following three types:
              \begin{enumerate}
                  \item For all non-exceptional
                        \(\nu\in\calJ_\lambda^{(n)}\), the term is
                        \(\Phi_{\lambda,\nu}^{(n)}(Y,h,r) S_\lambda(h,u_{\lambda,\nu})\).

                  \item For \((\lambda,\nu)=(n-2,n)\), the term is the finite part
                        involving \(S_{n-2}(h,m-1)\) and \(S'_{n-2}(h,m-1)\).

                  \item If \(\lambda\) is even,
                        \(\disc(h)\in\bbQ^{\times2}\), and
                        \(\nu=m+\lambda/2\), the term is the finite part involving
                        \(\calP_h(u_{\lambda,\nu})\) and
                        \(\calP'_h(u_{\lambda,\nu})\).
              \end{enumerate}

        \item Suppose that \(T\not\geq0\).
              Let \(\lambda=\rank(T)\), and choose \(h\) and \(r\) as in
              Section~\ref{sec:nonzero-coefficients};
              let \((p,q)\) be the signature of \(h\), with \(q>0\).
              Then
              \[ A_T^{(n)}(Y)=\sum_{\nu=\lambda}^{n}A_{T,\nu}^{(n)}(Y). \]
              The summands are given by
              Proposition~\ref{prop:indefinite-index-coefficients}.
              Outside the square-discriminant case, the only possible
              contributing pairs are
              \[
                  \begin{array}{c|c}
                      q & (\lambda,\nu)                                               \\ \hline
                      1 & (\lambda,\lambda+1),\ (n-2,n),\ (n-3,n),\ (n-4,n-1),\ (1,n) \\[2pt]
                      2 & (n-2,n)
                  \end{array}
              \]
              If \(n=10\), the first row also includes \((2,7)\), \((3,8)\), \((4,9)\), and \((5,10)\).

              In the square-discriminant case, the only possible
              contributing pairs are
              \[
                  \begin{array}{c|c}
                      q & (\lambda,\nu)      \\ \hline
                      1 & (n,n) ,\ (n-4,n-2) \\
                      2 & (n-2,n-1)
                  \end{array}
              \]
              Moreover, every summand \(A_{T,\nu}^{(n)}(Y,s)\) is
              holomorphic at \(s=0\).
    \end{enumerate}

    The formulas are independent of the auxiliary choice of the pair
    \((h,r)\).
\end{thm}

\begin{proof}
    The formula for the constant term is
    Theorem~\ref{thm:constant-term}.
    The assertion for positive semi-definite non-zero indices is exactly
    Proposition~\ref{prop:nonzero-coefficient-contribution}.
    The assertion for indices that are not positive semi-definite is exactly
    Proposition~\ref{prop:indefinite-index-coefficients}.
\end{proof}

\appendix
\section{The Middle-Rank Terms in Degree 6}
\label{app:degree-six-middle-rank}

In this appendix,
we record the cancellation of the middle-rank terms which occurs only in
degree \(6\).
Thus \(m=3\),
\(n=6\),
and \(k=4\).
In this case,
the terms \(C_3(Y,s)\) and \(C_4(Y,s)\) also have order \(0\) at \(s=0\).
\begin{prop}\label{prop:degree-six-middle-rank-cancellation}
    Assume that \(m=3\), so that \(n=6\) and \(k=4\).
    Then
    \[ C_3(Y,0)+C_4(Y,0)=0. \]
\end{prop}

\begin{proof}
    We first relate the two completed Koecher--Maass zeta functions
    which occur in the computation.
    Applying Proposition~\ref{prop:xi-special-values} with
    \(M=6\), \(\nu=3\), and \(\mu=2\), we get
    \begin{equation}\label{eq:xi-three-via-xi-two} \Xi_3^{(6)}(S,2)=-\frac13\det(S)^{-3/2}\Xi_2^{(6)}\left(S^{-1},\frac32\right). \end{equation}
    By the definition of \(\Xi_2^{(6)}\) and
    Lemma~\ref{lem:complementary-rank-duality},
    \begin{align*}
        \Xi_2^{(6)} \left(S^{-1},\frac32\right) & =\frac92\xi(3)\xi(2)\zeta_2^{(6)} \left(S^{-1},\frac32\right)          \\
                                                & =\frac92\xi(3)\xi(2)\det(S)^{3/2}\zeta_4^{(6)} \left(S,\frac32\right).
    \end{align*}
    Hence
    \begin{equation}\label{eq:xi-three-zeta-four} \Xi_3^{(6)}(S,2)=-\frac32\xi(3)\xi(2)\zeta_4^{(6)}\left(S,\frac32\right). \end{equation}
    On the other hand, the definition of \(\Xi_4^{(6)}\) gives
    \begin{equation}\label{eq:xi-four-zeta-four} \Xi_4^{(6)}\left(S,\frac32\right)=\frac{9}{32}\xi(3)\xi(2)\zeta_4^{(6)}\left(S,\frac32\right). \end{equation}
    Combining \eqref{eq:xi-three-zeta-four} and
    \eqref{eq:xi-four-zeta-four}, we obtain
    \begin{equation}\label{eq:degree-six} \Xi_4^{(6)}\left(S,\frac32\right)=-\frac{3}{16}\Xi_3^{(6)}(S,2). \end{equation}

    We now compute \(C_3(Y,0)\) and \(C_4(Y,0)\).
    For \(\nu=3\), formula \eqref{eq:Cnu-first-column} gives
    \[ C_3(Y,s)=8\pi^6\frac{\Gamma_3(2s+2)}{\Gamma_3(s)\Gamma_3(s+4)} S_3(0_3,2s+4)\det(Y)^s\zeta_3^{(6)}(2Y,2s+2). \]
    At \(s=0\), the relevant Laurent expansions are
    \[ \frac{\Gamma_3(2s+2)}{\Gamma_3(s)\Gamma_3(s+4)}=\frac{s^2}{90\pi^2}+O(s^3), \]
    \[ S_3(0_3,2s+4)=\frac{1}{2s}\frac{\zeta(3)}{\zeta(4)\zeta(6)}+O(1), \]
    and
    \[ \zeta_3^{(6)}(2Y,2s+2)= -\frac{\Xi_3^{(6)}(2Y,2)}{6s\,\xi(4)\xi(3)\xi(2)}+O(1). \]
    Therefore
    \[ C_3(Y,0)=-A_3\Xi_3^{(6)}(2Y,2), \]
    where
    \[ A_3=\frac{\pi^6}{135}\frac{\zeta(3)}{\zeta(4)^2\zeta(6)\xi(3)\xi(2)}. \]

    For \(\nu=4\), formula \eqref{eq:Cnu-first-column} gives
    \[ C_4(Y,s)=16\pi^{10} \frac{\Gamma_4\left(2s+\frac32\right)}{\Gamma_4(s)\Gamma_4(s+4)} S_4(0_4,2s+4)\det(Y)^s\zeta_4^{(6)}\left(2Y,2s+\frac32\right). \]
    The corresponding Laurent expansions are
    \[ \frac{\Gamma_4\left(2s+\frac32\right)}{\Gamma_4(s)\Gamma_4(s+4)} =\frac{s}{180\pi^4}+O(s^2), \]
    \[ S_4(0_4,2s+4)=-\frac{1}{8s}\frac{\zeta(3)}{\zeta(4)^2\zeta(6)}+O(1), \]
    and
    \[ \zeta_4^{(6)}\left(2Y,2s+\frac32\right)= \frac{32}{9\xi(3)\xi(2)}\Xi_4^{(6)}\left(2Y,\frac32\right) +O(s). \]
    Thus
    \[ C_4(Y,0)=-A_4\Xi_4^{(6)}\left(2Y,\frac32\right), \]
    where
    \[ A_4=\frac{16\pi^6}{405}\frac{\zeta(3)}{\zeta(4)^2\zeta(6)\xi(3)\xi(2)} =\frac{16}{3}A_3. \]
    Applying \eqref{eq:degree-six} with \(S=2Y\), we conclude that
    \begin{align*}
        C_3(Y,0)+C_4(Y,0) & =-A_3\Xi_3^{(6)}(2Y,2) -A_4\Xi_4^{(6)}\left(2Y,\frac32\right) \\
                          & =-A_3\Xi_3^{(6)}(2Y,2)+A_3\Xi_3^{(6)}(2Y,2)                   \\
                          & =0.
    \end{align*}
\end{proof}

\bibliographystyle{amsplain}
\bibliography{small_weight}
\end{document}